\def\btimes        {\times} 
\def\CbC           {\ensuremath{\hspace{1.1pt}\overline{\mathcal C}\hspace{.6pt}{\boxtimes}%
                    \hspace{1.8pt}\mathcal C}\xspace}
\def\CopC          {{\ensuremath{\mathcal C^{\mathrm{op}}\hspace*{.8pt}{\times}\hspace*{1.5pt}\mathcal C}}}
\def\DopD          {{\ensuremath{\mathcal D^{\mathrm{op}}\hspace*{.8pt}{\times}\hspace*{1.5pt}\mathcal D}}}
\def\End           {{\ensuremath{\mathrm{End}}}}
\def\Hom           {{\ensuremath{\mathrm{Hom}}}}
\def\HomC          {{\ensuremath{\mathrm{Hom}_{\mathcal C}}}}
\def\id            {\mbox{\sl id}} 
\def\idsm          {\mbox{\footnotesize\sl id}}
\def\idxv          {\ensuremath{\id_{x^\vee_{}}}}
\def\idxvx         {\ensuremath{\id_{x^\vee_{}\otimes x}}}
\def\idxvxvv       {\ensuremath{\id_{x^\vee_{}\otimes x^{\vee\vee}_{}}}}
 
\def\one           {{\bf1}}
\def\oti           {\,{\otimes}\,}
\def\resizeleftmargin{\def\leftmargini{1.04em}~\\[-1.21em]}
\def\Times         {\,{\times}\,}
\def\Vect          {\ensuremath{\mathcal V\hspace*{-1pt}\mbox{\sl ect}}}

\documentclass[12pt]{article}
\usepackage{latexsym, amsmath, amsthm, amsfonts, enumerate, amssymb, xspace}
 \usepackage{bbm}
        \usepackage{fancybox}
\usepackage[mathscr]{eucal}
\usepackage{graphicx} \usepackage{rotating}
\usepackage{epstopdf,hyperref}

\usepackage{tikz}
  \usetikzlibrary{shapes.geometric}
\usepackage{tikz-cd}
  \usetikzlibrary{matrix,arrows}

\newtheorem{thm}{Theorem}

\newtheorem{lemma}[thm]{Lemma}
\newtheorem{prop}[thm]{Proposition}
\newtheorem{coro}[thm]{Corollary}
\theoremstyle{definition}
\newtheorem{defi}[thm]{Definition}
\newtheorem{rem}[thm]{Remark}

\begin{document}

\thispagestyle{empty} 
\begin{flushright}
   {\sf ZMP-HH/16-6}\\
   {\sf Hamburger$\;$Beitr\"age$\;$zur$\;$Mathematik$\;$Nr.$\;$589}\\[2mm]
   April 2016
\end{flushright}
   \bigskip

         \vskip 2.0em
\begin{center}\Large
{\bf COENDS IN CONFORMAL FIELD THEORY}
\end{center}\vskip 2.2em
\begin{center}
  J\"urgen Fuchs $^{a}$
  ~~and~~ Christoph Schweigert $^b$
\end{center}

\vskip 9mm

\begin{center}
  $^a$ Teoretisk fysik, \ Karlstads Universitet
  \\Universitetsgatan 21, \ S\,--\,\,651\,88 Karlstad
 \\[7pt]
  $^b$ Fachbereich Mathematik, \ Universit\"at Hamburg\\
  Bereich Algebra und Zahlentheorie\\
  Bundesstra\ss e 55, \ D\,--\,20\,146\, Hamburg
\end{center}

\vskip 3em

   \bigskip

\noindent{\sc Abstract}
\\[3pt]
The idea of ``summing over all intermediate states'' that is central for implementing
locality in quantum systems can be realized by coend constructions. In the concrete 
case of systems of conformal blocks for a certain class
of conformal vertex algebras, one deals with coends in functor categories. Working with
these coends involves quite a few subtleties which, even though they have in principle 
already been understood twenty years ago \cite{lyub6,lyub11}, have not been 
sufficiently appreciated by the conformal field theory community.

  \newpage

%%%%%%%%%%%%%%%%%%%%%%%%%%%%%%%%%%%%%%%%%%%%%%%%%%%%%%%%%%%%%%%%%%%%%%%%%%%%%%%%%%

\section{Coends in mathematics and physics} 

In this note we discuss the use of coends in conformal field theories.
The notion of a coend is a standard concept in category theory; we review it briefly.
In the context of conformal field theory, the categories in which we consider coends
are related to representation categories of conformal vertex algebras. Conformal 
vertex algebras give rise to bundles of conformal blocks. These constitute the 
building blocks for correlators in conformal field theories. Moreover, they 
have a rich mathematical structure with deep links to many other fields of mathematics,
including representation theory and algebraic geometry. Our main goal in this note
is to provide pertinent categorical tools for the description of (aspects of) conformal
blocks and for constructions involving them. Both this description and the tools have 
been known for more than twenty years \cite{lyub6,lyub11}.  However, in our opinion,
Lyubashenko's theory has not been sufficiently appreciated. The present contribution
will hopefully make some aspects of them more easily accessible.

\subsection{Definition and examples}

A \emph{coend} is a colimit of a functor $G\colon \CopC \,{\to}\, \mathcal D$ -- an object $D$ of 
$\mathcal D$ equipped with a universal dinatural transformation from $G$ to $D$. One denotes 
the coend of $G\colon \CopC \,{\to}\, \mathcal D$, and also the underlying object $D$, by 
$\int^{c\in\mathcal C \!} G(c,c)$. There is also the obvious dual notion of an \emph{end}
as a limit over $G$, denoted by $\int_{c\in\mathcal C} G(c,c)$.

If the coend of a functor exists, then it is unique up to unique isomorphism. If $\mathcal D$
is co\-com\-plete, then the coend does exist; it can then be obtained
as the coequalizer of the mor\-phisms
  \begin{equation}
  \coprod_{f\in\HomC(c',c'')} \! \! \! G(c',c'')
  \begin{array}c ~\\[-11pt] 
  \xrightarrow{\,~s~~} \\[-7pt] \xrightarrow[\,~t~~]{\phantom{~s~}} \end{array}
  \coprod_{c\in\,\mathcal C} G(c,c)
  \label{coequrepn}
  \end{equation}
whose values at the $f$th summand are $s_f \,{=}\, G(f,\id)$ and
$t_f \,{=}\, G(\id,f)$, respectively (see e.g.\ \cite[Sect.\,V.1]{MAy4}).
In short, the coend of $G$ is a quotient of $\coprod G(c,c)$ on which 
the natural left and right actions of $G$ on any morphism $f$ in $\mathcal C$ coincide. The dinatural 
transformation is induced from the canonical maps of the coproduct and the quotient.

\medskip

(Co)limits, and thus (co)ends, are an enormously useful tool. 
Let us illustrate this with a few examples:
\resizeleftmargin \begin{itemize}
 
\item
Let $\mathcal C \,{=}\, \mathcal D=\Vect$ be the category of finite-dimensional vector 
spaces over a field $\Bbbk$, and let $G\colon \CopC \,{\to}\, \mathcal D$ be the inner Hom, 
i.e.\ $G(V,W) \,{:=}\, \Hom(V,W)$. Then the coend is the ground field $\Bbbk$ and the 
dinatural transformation $G(V,V) \,{=}\, \End(V)\,{\to}\, \Bbbk$ is given by the trace. This 
encodes simultaneously the cyclicity of the trace and a universal property of the trace.

\item
More generally, if $\mathcal C \,{=}\, \mathcal D$ is the category $H$-mod of 
fini\-te-dimensional modules over a finite-dimensional Hopf algebra $H$, then the coend of the
inner Hom is the coadjoint $H$-module $H^*_{\rm coad}$ with $M^*\oti M\,{\to}\, H^*$ given by
$\overline m\oti m \,{\mapsto}\, (h\,{\mapsto}\, \langle \overline m,hm\rangle)$.
In case $H$ is semisimple, the coadjoint module decomposes into a direct sum involving
simple $H$-modules according to
  $$ 
  H^*_{\rm coad} = \bigoplus_{S_i\in\pi_0(H\mbox{\footnotesize -}\mathrm{mod})}
  S_i^\vee\otimes S_i^{} \,.
  $$
Actually, the coend 
  \begin{equation}
  \label{eq:defL}
  L := \int^{c\in\mathcal C}\! c^\vee \oti c
  \end{equation}
of the inner Hom functor exists for any finite tensor category $\mathcal C$,
and indeed it can be seen as a generalization of the coadjoint module
(e.g.\ \cite{shimi9} $\Hom_{\mathcal C}(L,\one)$ and $\Hom_{\mathcal C}(\one,L)$ 
are analogues of the center of $H$ and of the space of class functions, respectively).

\item
The geometric realization of a simplicial set $F\colon \Delta^{\rm op} \,{\to}\,\mathrm{Set}$
can be regarded as a coend 
  $$
  \int^{n\in\Delta}\! F(n)\times \sigma(n) \,,
  $$
where $\sigma(n)$ is the standard $n$-simplex in $\mathbb{R}^{n+1}$
and the set $F(n)$ is endowed with the discrete topology. This example may be seen
as an archetype of a coend that implements a gluing modulo relations.

\item
Let $R$ be an associative unital ring and $\mathcal C \,{=}\, {*}\,/\!/R$ be the 
category with a single object $*$ and endomorphisms given by $R$. Then an
additive functor $F_M\colon \mathcal C \,{\to}\, \mathcal Ab$ 
with values in abelian groups amounts to a left $R$-module $N$, while an additive functor
$G_N \colon \mathcal C^{\mathrm{op}}\,{\to}\, \mathcal Ab$ amounts to a right $R$-module $N$. 
The coend $\int G_N\,{\times}\, F_M \,{\in}\, \mathcal Ab$ is then the relative 
tensor product $N \,{\otimes_R}\, M$.

\item
Provided that $\mathcal C$ is essentially small, the set of natural transformations 
between any two functors $F,G\colon \mathcal C \,{\to}\, \mathcal D$ can be expressed as an end:
  $$
  \mathrm{Nat}(F,G) = \int_{\!c \in \mathcal C} \Hom_{\mathcal D}(F(c),G(c)) \,.
  $$

\end{itemize}

{}From the second example in this list, we see that (co)ends supply a tool for constructing 
objects, in fact objects with additional structure. That structure embraces global 
information about the underlying category $\mathcal C$ and the functor $G$. Specifically,
if $\mathcal C$ is the representation category of a conformal vertex algebra whose 
representation category is a finite braided tensor category, then the coend \eqref{eq:defL}
has the additional structure of a Hopf algebra in $\mathcal C$ \cite{lyub3}. As we 
will recall in this contribution, this Hopf algebra plays an important role in one
convenient description of conformal blocks for higher genus Riemann surfaces.

A natural transformation $F \,{\to}\, G$ induces
morphisms of coends. In this way, coends yield functors; prominent examples
are the functor of geometric realization from simplicial sets to topological 
spaces and the functor of the relative tensor product.

\medskip

In many situations of interest the functor whose coend is considered depends on 
``parameters'', i.e.\ is a functor $G\colon \mathcal E\,{\times}\, \CopC \,{\to}\,
\mathcal D$ \cite[Ch.\,IX.7]{MAcl}. Coends can then be used as a tool to construct
functors $\mathcal E \,{\to}\, \mathcal D$.  Again we give a few examples:
\resizeleftmargin \begin{itemize}

\item
Assume that for two functors $H\colon \mathcal X \,{\to}\, \mathcal A$ and 
$K\colon \mathcal X \,{\to}\, \mathcal C$ the copowers $\HomC(K(x),c)\,{\cdot}\,H(x')$ 
exist in $\mathcal A$ for all objects $c\,{\in}\,\mathcal C$
and all objects $x,x'\,{\in}\,\mathcal X$. Then the existence 
of the coend $\int^{x\in\mathcal X \!}\HomC(K(x),c)\,{\cdot}\,H(x)$ for every $c\,{\in}\,\mathcal C$ 
implies \cite[Sect.\,X.4]{MAcl} that $H$ has a left Kan extension $\mathrm{Lan}_K H$ along $K$, 
and then $\mathrm{Lan}_K H(c)$ is given by this coend (and similarly with right Kan extensions
and ends). This covers a wide range of applications since, morally, every fundamental 
concept of category theory is a Kan extension \cite[Sect.\,X.7]{MAcl}.

\item
Let $\mathcal C$ be a rigid monoidal category, and assume that the coend
  $$
  Z(c) := \int^{b\,{\in}\,\mathcal C} \! b^\vee \oti c \oti b
  $$
(with $c^\vee$ the object dual to $c\,{\in}\,\mathcal C$) exists for all $c\,{\in}\,\mathcal C$. The
assignment $c \,{\mapsto}\, Z(c)$ then defines an endofunctor $Z$ of $\mathcal C$ that has a 
natural structure of a monad on $\mathcal C$, i.e.\ an algebra in the category of 
endofunctors. Moreover, $Z$ is even a quasitriangular Hopf monad, and the category
of $Z$-modules is braided isomorphic to the Drinfeld center $\mathcal Z(\mathcal C)$
of $\mathcal C$ \cite{daSt5,brVi5}.
In this way, we can think about an additional structure like the braiding or, more
generally, of a balancing on a bimodule category, as the structure of a module 
over the monad $Z$. Moreover, this
point of view allows for non-trivial calculations, see e.g.\ \cite{brVi5}.
\end{itemize}

%%%%%%%%%%%%%%%%%%%%%%%%%%%%%%%%%%%%%%%%%%%%%%%%%%%%%%%%%%%%%%%%%%%%%%%%%%%%%%%%%%

\subsection{Coends and quantum field theory}

We now explain how coends realize the idea of a \emph{sum over a complete set of 
intermediate states}, which is central for implementing locality in quantum physics.
We assume that the states are organized in terms of representations of some symmetry 
structure. For concreteness the reader may think of the situation that
chiral symmetries of a conformal field theory are encoded by a conformal vertex algebra,
so that the states are vectors in representations of the vertex algebra.  

For making the idea precise, we need in particular to specify what we mean by ``all'' states.
Summing over states in isomorphic representations would clearly be an over-parametrization.
But in fact \emph{all} morphisms, not only isomorphisms, must be taken into account.
And this is precisely afforded by taking a coend, as is seen from
its description \eqref{coequrepn} as a coequalizer in which morphisms are modded out.

We now consider concretely a class of quantum field theories for which important 
statements can be obtained with the help of categorical notions: two-di\-mensional
conformal field theories (CFTs) based on vertex algebras whose representation category 
$\mathcal C$ is braided monoidal and has certain finiteness properties and admits 
dualities. This class includes in particular all rational CFTs, for which $\mathcal C$ is a 
semisimple modular tensor category \cite{huan24},
but the use of coends allows us to treat non-se\-mi\-simple theories as well. 
(Because of the analytic properties of their conformal blocks
such theories go under the name of ``logarithmic'' conformal field theories.)
 
An important feature of a conformal vertex algebra is the fact that it gives
rise to system of conformal blocks, to which we refer as 
a chiral conformal field theory. The properties of these conformal blocks are,
for our purposes, conveniently encoded in a family of functors: given
a Riemann surface $\varSigma$ of genus $g$ with $p$ incoming and $q$ outgoing
punctures (in this context also called a world sheet), we have a functor
  $$
 \mathrm{Bl}(\varSigma) :~  \mathcal C^{\btimes (p+q)} \to  \Vect
  $$
which is covariant in the $q$ outgoing arguments and contravariant
in the $p$ incoming arguments. 
This system of functors comes with much additional structure, leading e.g.\ to 
actions of the mapping class groups Map$(\varSigma)$. 

In particular, also functors for different surfaces $\varSigma$ are related to one another.
It is indeed an old idea in CFT \cite{sono2,lewe3} that a pair-of-pants decomposition 
of the world sheet $\varSigma$ allows one to recover the conformal blocks from 
information about those for a small number of elementary world sheets through a suitable
sewing procedure. This sewing requires to perform an adequate sum over intermediate 
states. For those elementary world sheets the functors $\mathrm{Bl}(\varSigma)$ should,
in turn, be expressible in terms of Hom functors and of the tensor product of the 
representation category $\mathcal C$. The sum over intermediate states should be implemented 
by the fact that the functors $\mathrm{Bl}(\varSigma)$ for general world sheets are coends
of these functors. Sewing then amounts to considering coends in suitable functor categories 
$\mathcal F\hspace*{-2pt}un(\mathcal C^{\btimes q} \,{\btimes}\, (\mathcal C^{\mathrm{op}})^{\btimes p},\Vect)$.
We require that these functors are representable, so as to ensure an interpretation of 
the conformal blocks in terms of states. This leads us to consider coends within categories
of \emph{left exact} functors. As we will see, this also has many advantages when
manipulating the coends, see in particular Proposition \ref{prop:intHom-vs-L} below.

Remarkably enough, this program has been realized already more than twenty years ago
\cite{lyub6,lyub11} for categories that are modular (and thus in particular rigid), but 
not necessarily semisimple. A key role in this program is played by the coend  
$L \,{=}\, \int^{c\in\mathcal C}\! c^\vee \oti c$
of the inner Hom functor introduced in \eqref{eq:defL} above.
Specifically, if $\mathcal C$ is a finite tensor category, then this
coend $L$ exists as a coalgebra in $\mathcal C$, and if $\mathcal C$ is in addition braided,
then $L$ carries a natural structure of a Hopf algebra endowed with a non-zero left
integral and a Hopf pairing. This is an instance of the paradigm noted above that
a coend captures global properties of a category.

We adopt the terminology of \cite{KEly} to call a finite ribbon category \emph{modular}
iff the Hopf pairing on $L$ is non-de\-ge\-nerate. Then a semisimple modular finite 
ribbon category is the same as a (semisimple) modular tensor category.

\begin{rem}
Going beyond \cite{lyub6,lyub11}, which is concerned with chiral theories,
we are actually also interested in \emph{full} local -- as opposed to chiral -- CFT.
The relevant category over which a coend is to be taken is then the enveloping category
$\CbC$ of $\mathcal C$. However, $\CbC$ inherits from $\mathcal C$ all properties of interest
to us, such as being $\Bbbk$-linear, having a ribbon structure, or being finite.
As a consequence, the structures and results discussed in this note are directly relevant
both for describing chiral CFTs and for constructing full CFTs \cite{fuSc23}.
\end{rem}

\begin{rem}
Modularity of $\mathcal C$ in the sense used here is equivalent \cite{shimi10} to 
\emph{factorizability} of $\mathcal C$, i.e.\ to the functor 
$\CbC \,{\to}\, \mathcal Z(\mathcal C)$ from the enveloping category of $\mathcal C$
to its Drinfeld center that sends $u\,{\boxtimes}\,v$ to $u\,{\otimes}\,v$ with half-braiding
$ (c^{}_{u,-} \,{\otimes}\, \id_v^{}) \,{\circ}\, (\id_u^{} \,{\otimes}\, c^{-1}_{-,v}) $
being an equivalence. (In \cite{shimi10} also further equivalences are established.)
\end{rem}

Taken together, the coend constructions alluded to above allow one to express
the conformal blocks for a connected world sheet $\varSigma$ of genus $g$ 
as the $\Bbbk$-linear symmetric monoidal functor \cite{lyub11,lyub8}
  \begin{equation}
  \label{eq:genblock}
  \begin{array}{rrcl}
  \mathrm{Bl}(\varSigma) :~ & \mathcal C^{\btimes (p+q)} & \!{\longrightarrow}\! & \Vect
  \\{}\\[-10pt]
  & \mbox{${\huge\btimes}_\alpha$}\, u_\alpha^{} & \!{\longmapsto}\! &
  \Hom_{\mathcal C}\big( \one,
  \mbox{$\bigotimes_\alpha$}\, u_\alpha^\varepsilon \oti L^{\otimes g} \big)
  \end{array}
  \end{equation}
where $L$ is the Hopf algebra \eqref{eq:defL} 
and $u_\alpha^\varepsilon \,{=}\, u_\alpha^{}$ for an outgoing insertion while
$u_\alpha^\varepsilon \,{=}\, u_\alpha^\vee$ for an incoming insertion.
In the remainder of this note we provide various pertinent information about coends, and 
in particular results that are instrumental in arriving at the formula \eqref{eq:genblock}.
The main statements we will present are already available in the literature.
Besides including them for the sake of completeness, we also provide detailed proofs.

%%%%%%%%%%%%%%%%%%%%%%%%%%%%%%%%%%%%%%%%%%%%%%%%%%%%%%%%%%%%%%%%%%%%%%%%%%%%%%%%%%

\section{Some facts about specific coends} 

For a closed monoidal category, the inner Hom functor can be defined. For a rigid
monoidal category $\mathcal D$, it can be explicitly realized as $G\colon \DopD \to \mathcal D$ 
with $G\colon (c,d) \,{\mapsto}\, c^\vee \oti d$. Now recall
that a pivotal structure on a category $\mathcal D$ with duality $-^\vee$ is a monoidal
natural isomorphism $\pi\colon \mbox{\sl Id}_{\mathcal D} \,\,{\Rightarrow}\, -^{\vee\vee}$
between the identity functor of $\mathcal D$ and the double dual functor; we write the
components of $\pi$ as $\pi_x \,{\in}\, \Hom_{\mathcal D}(x,x^{\vee\vee})$. For the purposes
of conformal field theory, it is worth mentioning that a ribbon category comes
with a pivotal structure. 

\begin{lemma}\label{lem:end-coend}
Let $\mathcal D$ be a pivotal monoidal category and let $D \,{\in}\, \mathcal D$ be the 
object underlying the coend $\int^{x\in\mathcal D}\! G(x,x)$ of the inner Hom functor.
Then any choice of non-degenerate pairing $\varpi \,{\in}\, \Hom_{\mathcal D}(D\oti D,\one)$
induces on $D$ a structure of an end for the inner Hom. {\rm(}And vice versa, given a 
structure of an end on an object $E \,{\in}\, \mathcal D$, $\varpi$ induces on $E$ 
a structure of a coend.{\rm)}
\end{lemma}

\begin{proof}
Denote by $d$ and $b$ the evaluation and coevaluation for the duality of $\mathcal D$
and by $\imath$ with $\imath_x \,{\in}\, \Hom_{\mathcal D}(x^\vee{\otimes} x,D)$ the dinatural
transformation for the coend $D \,{=}\, \int^{x\in\mathcal D}\! G(x,x)$.
Define a family $\jmath$ of morphisms by
  $$
  \jmath_x := (\varpi \oti \idxv \oti \pi_x^{-1}) \circ
  (\id_D \oti \imath_x \oti \idxvxvv) \circ (\id_D \oti b_{x^\vee_{}\otimes x})
  ~\in \Hom_{\mathcal D}(D,x^\vee \oti x) \,.
  $$
By naturality of the duality and the pivotal structure and dinaturality of $\imath$,
also the family $\jmath$ is dinatural.
\\
Now assume that $\tilde\jmath$ is another dinatural transformation from some object
$y$ of $\mathcal D$ to the functor $G$. Then by an analogous argument as for $\jmath$, the 
family $\tilde\imath$ with
  $$
  \begin{array}{rl}
  \tilde\imath_x := (\id_{y^\vee_{}} \oti d_{x^\vee_{}\otimes x}) \circ
  (\id_{y^\vee_{}} \oti \idxv \oti \pi_x^{-1} \oti \idxvx)
  \\{}\\[-8pt] 
  \circ\, (\id_{y^\vee_{}} \oti [ \tilde\jmath_x \,{\circ}\, \pi_y^{-1} ] \oti \idxvx) \circ 
  (b_{y^\vee_{}} \oti \idxvx) & \in \Hom_{\mathcal D}(x^\vee \oti x,y^\vee)
  \end{array}
  $$
furnishes a dinatural transformation from $G$ to $y^\vee$. By the universal property 
of the coend $(D,\imath)$ there thus exists a unique morphism $\kappa \,{\in}\,
\Hom_{\mathcal D}(D,y^\vee)$ such that $\tilde\jmath \,{=}\, \kappa \,{\circ}\, \imath$.
Combining with the previous results, this implies that the morphism
  $$
  \lambda := (d_{y^\vee_{}} \oti \id_D) \circ (\pi_y \oti \kappa \oti \id_D)
  \circ (\id_y \oti \overline\varpi) ~\in \Hom_{\mathcal D}(y,D) \,,
  $$
with $\overline\varpi$ the non-degenerate copairing that is side-inverse to $\varpi$,
satisfies the property that is required for the universality of $(D,\jmath)$ as an
end of the functor $G$.
Finally, to see that $\lambda$ is unique with this property, assume that there exists
a morphism $\tilde\lambda$ with the same property. Inverting the arguments above then 
shows that there is another $\tilde\kappa \,{\in}\, \Hom_{\mathcal D}(D,y^\vee)$ with the same property
as $\kappa$. Universality of the coend implies that $\tilde\kappa \,{=}\, \kappa$, which in
turn gives $\tilde\lambda \,{=}\, \lambda$.
\end{proof}

In conformal field theory, an example to which Lemma \ref{lem:end-coend} applies
is the coend $L \,{\in}\, \mathcal C$ defined in \eqref{eq:defL} in case that $\mathcal C$
is modular. The required non-degenerate pairing is then a Hopf pairing on
the Hopf algebra $L$. Another situation of interest is that the object $D$ carries 
a structure of a Frobenius algebra. It is known \cite{fuSs3,fuSs6,fuSc23}
that the space of bulk fields of a CFT must admit a structure of a commutative
Frobenius algebra in the enveloping category \CbC, and candidates for such Frobenius 
algebras are given \cite{fuSs6,fuSs5} by the coends $\int^{c\in\mathcal C}\! 
c^\vee \,{\boxtimes}\, \omega(c)$ with $\omega$ a braided auto-equivalence of $\mathcal C$.
(For the existence of Frobenius structures on such objects, see \cite[Thm.\,6.1]{shimi7}.)

\medskip

When dealing with the morphism spaces that arise in the construction of 
conformal blocks, a non-degenerate pairing as required in Lemma \ref{lem:end-coend} 
does not exist, in general. To obtain the desired results we then 
work with coends rather than with ends. This allows us to invoke
the following result, which is a reformulation of Lemma B.1 of \cite{lyub11} and
for which there is no counterpart when using ends. 

\begin{prop}\label{prop:deltaX}
Let $\mathcal D$ be a $\Bbbk$-linear category and $G\colon \mathcal D \,{\to}\, \Vect$ 
a $\Bbbk$-linear functor. For any object $b\,{\in}\,\mathcal D$, the coend of the functor
  $$ 
  G(-) \,{\otimes_{\Bbbk}}\, \Hom_{\mathcal D}(-,b) :\quad
  \mathcal D \Times \mathcal D^{\mathrm{op}} \to \Vect
  $$ 
can be realized as the vector space $G(b)$ with the family of linear maps
  $$ 
  i_u: \quad G(u) \,{\otimes_{\Bbbk}}\, \Hom_{\mathcal D}(u,b) \, \ni\, w \oti f
  \,\longmapsto\, G(f).w \,\,{\in}\, \, G(b) 
  $$
for $u \,{\in}\, \mathcal D$ as a dinatural transformation. In particular, the coend exists.
We write
  $$
  \int^{d\in\mathcal D}\! G(d) \,{\otimes_{\Bbbk}}\, \Hom_{\mathcal D}(d,b) \,\cong\, G(b) \,.
  $$
\end{prop}

\begin{proof}
For any morphism $g \,{\in}\, \Hom_{\mathcal D}(u,v)$ the linear maps
  $$  
  \begin{array}{ll}
  i_u \circ (\id_{G(u)} \oti g^*): & w \oti f \longmapsto \big(G(f \,{\circ}\, g)\big)\,.w
  \qquad{\rm and}
  \\[-1.3em]\\[4mm]
  i_v \circ (G(g) \oti \id^*): & w \oti f \longmapsto G(f).\big(G(g).w \big)
  \end{array}
  $$  
are equal, hence the family $(i_u)$ is indeed dinatural.
To show the universal property of the coend, let $j_u \colon G(u) \,{\otimes_{\Bbbk}}\,
\Hom_{\mathcal D}(u,b) \,{\to}\, W$ be any dinatural transformation to a
vector space $W$. Then for any $g \,{\in}\, \mathcal D$ consider the diagram
  \begin{equation}
  \label{eq:xyp1X}
  \begin{tikzcd}
  ~ & G(u) \,{\otimes_{\Bbbk}}\, \End_{\mathcal D}(b) \ar{dl}[left,yshift=3pt]
  {\idsm_{G(u)}\otimes g^*} \ar{dr}[xshift=-3pt]{G(g)\oti\idsm^*} & ~
  \\
  G(u) \,{\otimes_{\Bbbk}}\, \Hom_{\mathcal D}(u,b) \ar{r}{i_u} \ar{dr}[left,yshift=-7pt,xshift=3pt]{j_u^{}}
  & G(b) \ar{d}{\kappa} 
  & G(b) \,{\otimes_{\Bbbk}}\, \End_{\mathcal D}(b) \ar{dl}[yshift=2pt]{j_b} \ar{l}[left,yshift=6pt]{i_b}
  \\
  ~& W & ~
  \end{tikzcd}
  \end{equation}
  where the linear map $\kappa\colon G(b) \,{\to}\, W$ is defined as
$\kappa(w) \,{:=}\, j_b(w \oti \id_b)$.
The outer square of the diagram \eqref{eq:xyp1X} commutes by dinaturalness of the 
family $(j_u)$, and the upper square commutes by dinaturalness of $(i_u)$. Further we have
  \begin{equation}
  \label{eq:kappa====}
  \kappa \,{\circ}\, i_u(w{\otimes}f) \equiv \kappa(G(f).w) = j_b(G(f).w \oti \id_B)
  = j_u(w \oti f^*_{}(\id_b)) = j_u(w{\otimes}f) 
  \end{equation}  
for all $w\oti f \,{\in}\, G(u) \,{\otimes_{\Bbbk}}\, \Hom_{\mathcal D}(u,b)$.
Hence $\kappa \,{\circ}\, i_u \,{=}\, j_u$, i.e.\ the left hand triangle in \eqref{eq:xyp1X} 
commutes. (The right triangle, which is just a specialization
of the left one, obtained by setting $u$ to $b$, then commutes as well.)
Thus the map $\kappa$ satisfies the equalities needed for the vector space $G(b)$ 
to have the universal property of the coend.
Uniqueness of $\kappa$ follows immediately by reading the sequence \eqref{eq:kappa====}
of equalities from right to left and noticing that the elements $i_b(w{\otimes} f)$ span 
the vector space $G(b)$.
\end{proof}

We refer to the property of the Hom functor
asserted by Proposition \ref{prop:deltaX} as the \emph{delta function property}.
In the special case that $G$ is a Hom functor as well, we have

\begin{coro}\label{coro:deltaprop}
Let $\mathcal D$ be a $\Bbbk$-linear category. For any pair of objects 
$u,v \,{\in}\, \mathcal D$ the coend of the func\-tor $\Hom_{\mathcal D}(u,-) 
\,{\otimes_{\Bbbk}}\, \Hom_{\mathcal D}(-,v)$ exists.  As an object it is
  \begin{equation}
  \label{eq:deltaprop}
  \int^{x\in\mathcal D}\! \Hom_{\mathcal D}(u,x) \,{\otimes_{\Bbbk}}\, \Hom_{\mathcal D}(x,v)
  = \Hom_{\mathcal D}(u,v)
  \end{equation}
and the dinatural transformation is given by composition,
  $$  
  i_x(f,g) = g \circ f 
  $$  
for $f \,{\in}\, \Hom_{\mathcal D}(u,x)$ and $g \,{\in}\, \Hom_{\mathcal D}(x,v)$.
\end{coro}

The delta function property can be used to perform a non-trivial check of
the idea that coends implement sewing of conformal blocks. We postulate that
for a surface of genus zero, conformal blocks are expressible in terms
of the tensor product and the Hom functor. Concretely, if a surface $\varSigma$ 
of genus 0 has $p$ incoming boundary circles with objects $u_1,...\,, u_p$
and $q$ outgoing boundary circles with objects  $\tilde u_1,...\,,\tilde u_q$,
the conformal blocks should be 
  $$
  \mathrm{Bl}(\varSigma)(u_1,...\,, u_p; \tilde u_1,...\,, \tilde u_q)
  = \Hom_{\mathcal C}(u_1 \,{\otimes}\cdots {\otimes}\, u_p,
  \tilde u_1 \,{\otimes}\cdots {\otimes}\, \tilde u_q) \,.
  $$
Suppose now that we sew two distinct genus-0 surfaces $\varSigma_1,\varSigma_2$ to 
a connected surface of genus zero. Sewing amounts to identify 
an outgoing boundary component of $\varSigma_1$ with an incoming one of $\varSigma_2$
(or vice versa); these boundary circles have to carry identical field insertions $x$,
and over these a coend is to be taken. Abbreviating the corresponding block spaces by
$\mathrm{Bl}(\varSigma_1)(u;\tilde u, x)$ and by $\mathrm{Bl}(\varSigma_2)(x, v;\tilde v)$,
respectively, we thus obtain
  $$
  \begin{array}{ll}
  \displaystyle
  \int^{x\in\mathcal C}\!\! \mathrm{Bl}(\varSigma_1)(u;\tilde u, x) \,{\otimes_{\Bbbk}}\,
  \mathrm{Bl}(\varSigma_2)(x \oti v;\tilde v) 
  \!\! & \displaystyle
  \cong \int^{x\in\mathcal C}\!\! \Hom_{\mathcal C} (u,\tilde u \oti x)
  \,{\otimes_{\Bbbk}}\, \Hom_{\mathcal C} (v \oti x;\tilde v)
  \\{}\\[-9pt] & \displaystyle
  \cong\, \Hom_{\mathcal C} ( u \oti v , \tilde u \oti \tilde v)
  \,\cong\, \mathrm{Bl}(\varSigma) (u,v ; \tilde u,\tilde v) \,.
  \end{array}
  $$
Hence the result of the sewing is indeed compatible with our postulate for the
conformal blocks of genus-$0$ surfaces.

%%%%%%%%%%%%%%%%%%%%%%%%%%%%%%%%%%%%%%%%%%%%%%%%%%%%%%%%%%%%%%%%%%%%%%%%%%%%%%%%%%

\section{Coends in functor categories}

As already pointed out, conformal blocks are functors. Moreover, conformal blocks for
general surfaces should be obtained from morphism spaces by summing over intermediate 
states, and thus by taking coends. For constructing conformal blocks we therfore need 
to consider coends with values in functor categories. A central fact in this respect is
the so-called parameter theorem for coends \cite[Sect.\,IX.7]{MAcl}.

\subsection{Parameter theorem for coends}

We start with a functor $G \,{=}\, G(?;-,-) \colon \mathcal D\Times 
\mathcal C^{\mathrm{op}}\Times\mathcal C \,{\to}\, \mathcal E$. On the one hand, 
we may invoke adjunction in the bicategory of categories and reinterpret $G$ as a functor
  $$ 
  \widetilde G(-,-) := G(?;-,-): \quad
  \mathcal C^{\mathrm{op}}\Times\mathcal C \xrightarrow~ \mathcal F\hspace*{-2pt}un(\mathcal D , \mathcal E) \,.
  $$ 
If the coend of $\tilde G$ exists, it is an object in the functor category
$\mathcal F\hspace*{-2pt}un(\mathcal D , \mathcal E)$; we denote this functor by
  $$ 
  \big( \int^{c\in\mathcal C}\! \widetilde G(c,c)\, \big)\, (?): \quad \mathcal D \to 
  \mathcal E \,.
  $$
On the other hand  we may perform the following construction: For a fixed object
$d \,{\in}\, \mathcal D$, called a \emph{parameter} in this context, we obtain a
functor $G_d \,{:=}\, G(d;-,-) \colon \mathcal C^{\mathrm{op}}\Times\mathcal C \,{\to}\, \mathcal E$. Suppose its coend exists; it
is an object
  $$
  e_{d} := \int^{c\in\mathcal C}\! G_{d}(c,c) \,\in \mathcal E \,.
  $$
The fact that $G$ is functorial in $d$ as well implies that the assignment
  $$
  d \,\longmapsto\, e_d
  $$
defines a functor 
  $$
  \int^{c\in\mathcal C}\! G(?;c,c):\quad \mathcal D \,{\to}\, \mathcal E \,.
  $$
The parameter theorem for coends now states: 

\begin{thm}
Let $\mathcal D\Times \mathcal C^{\mathrm{op}}\Times\mathcal C \,{\to}\, \mathcal E$ be a 
functor. Then the functor 
  $$
  \int^{c\in\mathcal C}\! G(?;c,c):\quad \mathcal D \,{\to}\, \mathcal E
  $$
has a natural structure of a coend for the functor
  $$
  \tilde G:\quad \mathcal C^{\mathrm{op}}\Times\mathcal C \xrightarrow~ 
  \mathcal F\hspace*{-2pt}un(\mathcal D , \mathcal E) \,,
  $$
provided that all coends $\int^{c\in\mathcal C}\! G(d;c,c)$ exist. We write
  \begin{equation}
  \label{eq:G??xx}
  \int^{c\in\mathcal C}\! G(?;c,c) = \big( \int^{c\in\mathcal C}\! \widetilde G(c,c)\, \big)\, (?) \,.
  \end{equation}  
\end{thm}

%%%%%%%%%%%%%%%%%%%%%%%%%%%%%%%%%%%%%%%%%%%%%%%%%%%%%%%%%%%%%%%%%%%%%%%%%%%%%%%%%%

\subsection{Left exact coends and representability}

These rather general statements are, however, not directly suited for the application 
to conformal blocks we have in mind.
Indeed, in the construction of conformal blocks one encounters the situation that
a connected world sheet of genus $g$ is sewn to a world sheet of genus $g{+}1$.
This naturally leads to a functor which is a coend of the form $\int^{e\in\mathcal C}\!
\Hom_{\mathcal C}(-,- \oti S(e,e))$ with $S$ the inner End functor.
For applications to quantum field theory, it is desirable to express such a coend in the form
  $$
  \Hom(-,-\otimes L)
  $$
for some object $L$ that has a direct interpretation in a representation category $\mathcal C$,
so that the representation category organizes not only the incoming and outgoing states, 
but the intermediate states as well. One might call this requirement 
the existence of an operator calculus in terms of the underlying representation category.

Thus, roughly speaking, the idea is to ``pull the coend into the Hom functor.'' 
However, the Hom functor is continuous, and thus compatible with the end, which is a limit:
in the case of an end, we have for any functor $G\colon \CopC \to \mathcal D$ the equality
  \begin{equation}
  \label{eq:MAclIX5}
  \Hom_{\mathcal D}\big(d,\int_{c\in\mathcal C} G(c,c) \big)
  = \displaystyle \int_{c\in\mathcal C} \Hom_{\mathcal D}(d,G(c,c)) \,,
  \end{equation}  
and similarly, when $G$ appears in the first argument,
  \begin{equation}
  \Hom_{\mathcal D}\big(\int^{c\in\mathcal C}\! G(c,c),d \big)
  = \displaystyle \int_{c\in\mathcal C} \Hom_{\mathcal D}(G(c,c),d) \,,
  \end{equation}  
provided that the (co)end on the left hand side exists \cite[Sect.\,IX.5]{MAcl}.
In contrast, there are no similar equalities for the coends
$\int^{c} \Hom_{\mathcal D}(d,G(c,c))$ and $\int^{c} \Hom_{\mathcal D}(G(c,c),d)$.

As already explained, we want to have some representability of the block
functors within the representation category. For this reason,
the following result is important, as it singles out left exact functors:

\begin{prop}\label{prop:doSs1.10} {\rm \cite[Cor.\,1.10]{doSs}}
For $\mathcal C$ a finite $\Bbbk$-linear category, a functor $G\colon \mathcal C^{\mathrm{op}}
\,{\to}\, \Vect$ is representable, i.e.\ $ G(-) \,{\cong}\, \HomC(-,y) $ 
for some object $y$ of \,$\mathcal C$, iff the functor $G$ is left exact.
\end{prop}

Indeed, our construction starts with Hom functors, which are
left exact. We thus need a modification of the coend construction
for left exact functors. The following construction is due to \cite{lyub11}.

We are given a left exact $\Bbbk$-linear functor 
$G \,{=}\, G(?;-,-) \colon \mathcal D\Times \mathcal C^{\mathrm{op}}\Times\mathcal C \,{\to}\, \mathcal E$,
which we again reinterpret as a functor
  $$ 
  \widetilde G(-,-) := G(?;-,-): \quad
  \mathcal C^{\mathrm{op}}\Times\mathcal C \xrightarrow~ 
  \mathcal F\hspace*{-2pt}un(\mathcal D , \mathcal E) \,.
  $$ 
The ordinary coend $H \,{:=}\, \int^{c\in\mathcal C}\! \widetilde G(c,c)$ of this functor 
is a functor $H\colon \mathcal D \,{\to}\, \mathcal E$ with a dinatural family that is universal 
among all dinatural transformations to all functors. We now adapt our universality requirement:
the \emph{left exact coend} is a left exact functor
$H_{\rm l.ex.} \,{:=}\, \oint^{c\in\mathcal C}\! \widetilde G(c,c)\colon \mathcal D\to\mathcal E$
with a dinatural family of natural transformations with components
  $$
  \nu_{c;d}^{}:\quad G(d,c,c)\to H_{\rm l.ex.}(d)
  $$
that is universal among all dinatural transformation to all \emph{left exact} 
functors $\mathcal D \,{\to}\,\mathcal E$, rather than among
all dinatural transformation to all functors $\mathcal D \,{\to}\, \mathcal E$.
  
We summarize this in the following definition \cite{lyub11}:

\begin{defi} \label{def:oint}
For a left exact $\Bbbk$-linear functor $G\colon \mathcal D\Times\mathcal C\Times
\mathcal C^{\mathrm{op}}
\,{\to}\, \mathcal E$, the coend $\oint^{x}\! G(?;x,x)$ with values in the functor category
$\mathcal Lex(\mathcal D,\mathcal E)$ of left exact functors which is characterized by universality
among dinatural transformations to left exact functors is called the \emph{left exact coend}.
\end{defi}

We now assume that the category $\mathcal C$ is a finite rigid tensor category.
Thus in particular the coend
$ L \,{:=}\, \int^{c\in\mathcal C}\! c \oti c^\vee$ exists as an object of $\mathcal C$.
Working with left exact coends supplies a substitute for the identities \eqref{eq:MAclIX5},
namely the following statement, which is contained in the 
constructions given in Section 8.2 of \cite{lyub11}:

\begin{prop}\label{prop:intHom-vs-L}
Consider the $\Bbbk$-linear functor 
  $$
  \HomC(u,v\,{\otimes}-{\otimes}\,-^\vee) \colon\quad \mathcal C^{\mathrm{op}}\Times\mathcal C \to \Vect \,.
  $$ 
The family  
  $$
  (\id_v \oti \imath^L_c)_*^{}:\quad \HomC(u,v\,{\otimes}\, c \,{\otimes}\, c^\vee)\to \Hom(u,v\otimes L) 
  $$
of morphisms, with $(L;\imath^L)$ the coend $\int^{x\in\mathcal C}\!x \oti x^\vee$, 
is dinatural and endows the left exact functor 
  $$
  \HomC(-,-\oti L)\in \mathcal Lex(\mathcal C^{\mathrm{op}}\Times\mathcal C,\Vect)
  $$
with the structure of a left exact coend in the sense of Definition {\rm \ref{def:oint}}.
\\
In short,
  \begin{equation}
  \label{eq:oint=HomL}
  \oint^{x\in\mathcal C}\!\! \HomC(u,v\oti x \oti x^\vee)
  = \big(\, \HomC(u,v\oti L) \,; (\id_v \oti \imath^L)_*^{} \,\big) \,.
  \end{equation}  
\end{prop}

\begin{proof}
First note that the family $(\id_v \oti \imath^L)_*^{}$ is natural in $u$
(because pre- and post-composition commute) as well as in $v$ (because
$\id_v$ commutes through post-composition). Further, the family
$(\id_v \oti \imath^L)_*^{}$ of linear maps furnishes a dinatural transformation: 
Dinaturalness means that
  $$ 
  (\id_v \oti \imath^L_{x'})_*^{} \circ (\id_x \oti \id_{x'} \oti f^\vee{)}_*^{} 
  = (\id_x \oti \imath^L_x)_*^{} \circ (\id_x \oti f \oti \id_{x^\vee_{}}{)}_*^{} 
  $$ 
holds for any morphism $f \,{\in}\, \HomC(x',x)$; this property
follows directly from the dinaturalness of the family $\imath^L$.
\\[2pt]
To see the universal property, let $G(u;v)\colon \mathcal C^{\mathrm{op}}\Times\mathcal C \,{\to}\,
\Vect$ be a functor that for any choice of parameters $u,v\,{\in}\,\mathcal C$ is left exact,
and let $\imath^G$, with
  $$ 
  \imath^G_W \equiv \imath^{G(u;v)}_W\colon \quad
  \HomC(u,v\oti x \oti x^\vee) \longrightarrow G(u;v) \,,
  $$ 
be a dinatural family in the category of left exact functors from $\mathcal C^{\mathrm{op}}\Times\mathcal C$ to \Vect. 
In the sequel we regard $v$ as kept fixed and $u$ as a variable, 
whereby we deal with a left exact functor $G$ from $\mathcal C^{\mathrm{op}}$ to \Vect. 
According to Proposition \ref{prop:doSs1.10}, such a functor is representable, i.e.\ 
we have isomorphisms $\gamma_u \colon G(u) \,{\xrightarrow{~\cong~}}\, \HomC(u,y_G)$ 
for some $y_G\,{\in}\,\mathcal C$.  Dinaturalness of the family $\imath^G$
thus means that for any $f \,{\in}\, \HomC(x',x)$ the diagram 
  \begin{equation}
  \label{eq:dinat-iG}
  \begin{tikzcd}
  \HomC(u,v\oti x'\oti x^\vee) \ar{rr}{(\idsm_v\otimes f\otimes \idsm_{x^\vee_{}})_*^{}}
  \ar{d}{{(\idsm_v\otimes \idsm_{x'}\otimes f^\vee)}_*^{}}
  &~& \HomC(u,v\oti x\oti x^\vee) \ar{d}{\gamma_u\,\circ\,\imath^G_x}
  \\
  \HomC(u,v\oti x'\oti x'{}^{\vee}_{})  \ar{rr}{~~~~~~\imath^G_{x'}}
  &~& \HomC(u,y_G)
  \end{tikzcd}
  \end{equation}  
commutes for all $u \,{\in}\, \mathcal C$. 
By the Yoneda lemma this implies the existence of a commutative diagram
  $$ 
  \begin{tikzcd}
  v\oti x'\oti x^\vee \ar{rr}{\idsm_v\otimes f\otimes \idsm_{x^\vee_{}}}
  \ar{d}{\idsm_v\otimes \idsm_{x'}\otimes f^\vee}
  &~& v\oti x\oti x^\vee \ar{d}{}
  \\
  v\oti x'\oti x'{}^{\vee}_{} \ar{rr}{} &~& y_G
  \end{tikzcd}
  $$ 
in $\mathcal C$, for any $f \,{\in}\, \HomC(x',x)$, i.e.\ there is a dinatural 
family $\jmath_x\colon v \oti x\oti x^\vee \,{\to}\, y_G$ in $\mathcal C$.
The universal property of the coend $\int^x\! v \oti x\oti x^\vee$ then provides
us with a morphism $\int^x\! v \oti x\oti x^\vee \,{\xrightarrow{~~}}\, y_G$
-- the unique morphism which when composed with the dinatural family for the
coend gives the dinatural family $\jmath$. 
Using that the tensor product functor preserves colimits in $\mathcal C$,
we then also have a morphism $v \oti L \,{=}\, v \,{\otimes} \int^x\! x\oti x^\vee
\,{\xrightarrow{~~}}\, y_G$ or, by applying the Hom functor, a linear map
$\kappa^L_{G(u),v} \colon\, \HomC(u,v \oti L) \,{\xrightarrow{~~}}\, \HomC(u,y_G) \,{=}\, G(u)$.
In summary, and restoring $v$ as a parameter, we have obtained morphisms
  $$ 
  \kappa^L_{G(u;v)} : \quad \HomC(u,v \oti L) \longrightarrow G(u;v)
  $$ 
functorial in $u$. One can check that they are functorial in $v$ as well.
Moreover, by construction, these morphisms satisfy
$ i^{G(u;v)}_x \,{=}\, \kappa^L_{G(u;v)} \,{\circ}\, \big(\id_v \oti \imath^L_x \big)_{\!*} $,
i.e., invoking the representability isomorphisms $G(-) \,{\cong}\, \HomC(-,y_G)$, the triangle
  \begin{equation}
  \label{eq:kappa.tria}
  \begin{tikzcd}
  \HomC(u,v \oti L) \ar{d}[left]{\kappa^L_{G(u;v)}}
  & ~ & \HomC(u,v\oti x'\oti x^\vee) \ar{dll}[yshift=3pt]{i^{G(u;v)}_x}
  \ar{ll}[left,yshift=8pt,xshift=27pt]{(\mbox{\footnotesize\sl id}_v \otimes \imath^L_x{)}_*}
  \\
  \HomC(u,y_G)
  \end{tikzcd}
  \end{equation}  
commutes. To see that they are in fact the unique morphisms with this property,
invoke the Yoneda lemma to obtain a commutative triangle
  $$ 
  \begin{tikzcd}
  v \oti L \ar{d}[left]{\tilde\kappa^L_{G(u;v)}}
  & ~ & v\oti x\oti x^\vee \ar{dll}[yshift=2pt]{j_x}
  \ar{ll}[left,yshift=8pt,xshift=22pt]{\mbox{\footnotesize\sl id}_v \otimes \imath^L_x}
  \\
  y_G
  \end{tikzcd}
  $$ 
Regarding $v\oti L$ as the coend $\int^{x\in\mathcal C}\! v\oti x\oti x^\vee$, with dinatural
family $\id_v \oti \imath^L$, the morphism $\tilde\kappa^L_{G(u;v)}$ in this diagram
is uniquely determined, and hence so is 
$\kappa^L_{G(u;v)} \,{=}\, \big(\tilde\kappa^L_{G(u;v)}{\big)}_{\!*}$ in \eqref{eq:kappa.tria}.
This establishes the universal property of the dinatural transformation 
$(\id_v \oti \imath^L{)}_*$ and thus finishes the proof.
\end{proof}

\begin{rem}
Assume that, for $\mathcal C$ and $\mathcal E$ finite tensor categories and a functor
$G\colon \mathcal C^{\mathrm{op}} \Times \mathcal C \,{\to}\, \mathcal E$, the left exact coend
$\oint^{c\in\mathcal C}\! \Hom_{\mathcal E}(-,- \oti G(c,c))$ exists.
Then the coend $\int^{c\in\mathcal C}\! G(c,c)$ exists in $\mathcal E$ and we have
  $$
  \oint^{c\in\mathcal C}\! \Hom_{\mathcal E}(-,- \oti G(c,c))
  \,\cong\,\, \Hom_{\mathcal E}(-,- \oti \mbox{$\large\int$}^{c\in\mathcal C} G(c,c)) \,,
  $$
with the structural morphisms of the coends on the left and right hand side related via
the Yoneda lemma. This can be shown similarly as Lemma 3.1 in \cite{shimi7}: First 
one uses the fact that, $\mathcal E$ being finite, the left exact functor 
$(u,v)\,{\mapsto}\,\oint^{c\in\mathcal C}\! \Hom_{\mathcal E}(u,v \oti G(c,c))$
from $\mathcal E^{\mathrm{op}} \Times \mathcal E$ to $\Vect$ is representable 
and thereby provides an object underlying the coend. Then one invokes the Yoneda
lemma in a similar way as in the proof of Proposition \ref{prop:intHom-vs-L}
to obtain the structure morphisms for the coend.
\end{rem}

\begin{rem}
For any finite tensor category $\mathcal C$ the category
$\widehat{\mathcal C} \,{:=}\, \mathcal Lex(\mathcal C^{\mathrm{op}},\Vect)$ of $\Bbbk$-linear 
left exact functors from $\mathcal C^{\mathrm{op}}$ to \Vect\ has a 
monoidal structure given by convolution. Moreover, the general form of such 
a convolution tensor product (which is e.g.\ discussed in \cite[Sect.\,2]{daSt5})
simplifies, as a consequence of Proposition \ref{prop:doSs1.10}, to
  $$ 
  \big( G_1 \oti G_2 \big) (-) := \!\int^{u\in\mathcal C}\!\! G_1(-{\otimes}u^\vee)
  \,{\otimes_{\Bbbk}}\, G_2(u) \,\cong\, \HomC(-,y_1 \oti y_2)
  $$ 
for all $G_1,G_2 \,{\in}\, \widehat{\mathcal C}$ such that $G_1(-) \,{\cong}\, \HomC(-,y_1)$ and
$G_2(-) \,{\cong}\, \HomC(-,y_2)$ with $y_1,y_2 \,{\in}\,\mathcal C$.
\end{rem}

\begin{rem}
For any small $\Bbbk$-linear abelian rigid monoidal category $\mathcal C$, the assignment
$\mathcal C \,{\ni}\, x \mapsto \HomC(-,x)$ provides a full embedding of $\mathcal C$ into 
the functor category $\widehat{\mathcal C}$.
This embedding is an exact monoidal functor.
The category $\widehat{\mathcal C}$ admits arbitrary limits and colimits, and thus
contains all pro-objects and ind-objects of $\mathcal C$ as objects; see e.g.\
\cite[Sect.\,3.4]{lyub3} and \cite[Thm.\,5.40]{BUde}.
Various results involving the coend $L$ are still valid as long as it exists as an object
of the category $\widehat{\mathcal C}$ rather than even of $\mathcal C$ \cite{lyub11}.
Specifically, Proposition \ref{prop:intHom-vs-L} still holds in this broader setting. (In
its proof the object $y_G\,{\in}\,\mathcal C$ then needs to be replaced by an object
in $\widehat{\mathcal C}$, and diagrams in $\mathcal C$ like \eqref{eq:dinat-iG} turn into
diagrams in $\widehat{\mathcal C}$.)
\end{rem}

%%%%%%%%%%%%%%%%%%%%%%%%%%%%%%%%%%%%%%%%%%%%%%%%%%%%%%%%%%%%%%%%%%%%%%%%%%%%%%%%%%

\section{Fubini theorems}

A world sheet $\varSigma$ can have many different pair-of-pants decompositions. 
Conversely, many different sewings give rise to the same $\varSigma$ and, correspondingly,
the conformal blocks on $\varSigma$ can be described in many different ways as iterated
coends. On the other hand, sewing should be a local operation, and hence the order of
sewing should not matter. In order to obtain unique conformal blocks, such coends must 
therefore commute appropriately. This is indeed the case. We first quote
the standard Fubini theorem for coends:

\begin{prop}\label{prop:fubini} {\rm \cite[Ch.\,IX.7]{MAcl}} \\
Let $F\colon \mathcal C \Times \mathcal C^{\mathrm{op}} \Times \mathcal D \Times 
\mathcal D^{\mathrm{op}} \,{\to}\, \mathcal E$ be a functor
for which the coends $\int^{u\in\mathcal C}\! F(u,u,y,z)$ as well as the coends
$\int^{x\in\mathcal D}\! F(v,w,x,x)$ exist, for all $y,z\,{\in}\,\mathcal D$ and all 
$v,w\,{\in}\,\mathcal C$, respectively. Then there are unique isomorphisms
  $$ 
  \int^{u\in\mathcal C}\!\! \Big( \int^{x\in\mathcal D}\!\!\!\! F(u,u,x,x) \Big)
  \cong \!\int^{u\times x\,\in\,\mathcal C\times \mathcal D}\!\!\!\! F(u,u,x,x) 
  \cong \!\int^{x\in\mathcal D}\!\! \Big( \int^{u\in\mathcal C}\!\!\!\! F(u,u,x,x) \Big) \,.
  $$ 
In particular, each of these multiple coends exists.
\end{prop}

Since this formula resembles the commutativity of two integrations, 
it is referred to as the \emph{Fubini theorem} for iterated coends.
Invoking the Fubini theorem, the result \eqref{eq:oint=HomL} extends directly to multiple coends:

\begin{coro}
For any positive integer $g$ the functor 
$\HomC\big(u,v\oti (- {\otimes} -^{\!\vee})^{\otimes g} \big)\colon (\mathcal C{\times}
\mathcal C^{\mathrm{op}})^{\btimes g} \,{\to}\, \Vect$ has a $g$-fold left exact coend. It 
is given by
  $$ 
  \begin{array}{l} \displaystyle
  \oint^{(x_1 \btimes \,\cdots\, \btimes x_g)\,\in\,\mathcal C^{\btimes g}}\!\!
  \HomC(u,v\oti x_1^{} \oti x_1^\vee\oti \cdots \oti x_g^{} \oti x_g^\vee)
  \\[-1.3em]\\[1mm] \hspace{13.5em}
  = \big(\, \HomC(u,v\oti L^{\otimes g}) \,; (\id_v \oti (\imath^L)^{\otimes g}{)}_* \,\big)
  \,.  \end{array}
  $$ 
\end{coro}

\begin{proof}
We are free to replace the object $v \,{\in}\, \mathcal C$ in the proof of Proposition 
\ref{prop:intHom-vs-L} by $v \oti L$. We can therefore invoke Proposition 
\ref{prop:intHom-vs-L} together with the Fubini theorem, as
applied to left exact coends, to reduce the claim for some given $g$
with parameter object $v$ to the same claim for $g{-}1$ with parameter  
$v \oti L$. The assertion thus follows by induction.
\end{proof}

The Fubini theorem generalizes to situations in which two different types of coends,
one of them defined in terms of a suitable subcategory, are involved. In the context
of the construction of conformal blocks this happens when one describes pair-of-pants
decompositions of surfaces of higher genus. In that case
we deal with a functor category, and the relevant subcategory is the one of left exact 
functors. We formulate the statement directly for this situation.

\begin{prop}\label{thm:lyub11:B.2.} {\rm \cite[Thm.\,B.2]{lyub11}} \\
Let $\mathcal C_1, \mathcal C_2, \mathcal C_3$ and $\mathcal A$ be $\Bbbk$-linear abelian 
categories and $F\colon \mathcal C_2^{\mathrm{op}} \Times \mathcal C_1^{\mathrm{op}} \Times \mathcal C_1^{} \Times
\mathcal C_2^{} \Times \mathcal C_3^{} \,{\to}\, \mathcal A$ be a left exact $\Bbbk$-li\-near functor.
Assume that there exists a coend with parameters
  $$ 
  G(u,v,w) := \int^{x\in\mathcal C_1}\!\! F(u,x,x,v,w) \,,
  $$ 
with dinatural transformation $j_x(u,v,w)\colon F(u,x,x,v,w) \,{\to}\, G(u,v,w)$. Regard 
$G \,{=}\, G(?,\reflectbox{$?$},w)$ as a functor $G\colon \mathcal C_2^{\mathrm{op}} \Times
\mathcal C_2^{} \,{\to}\, \mathcal F\hspace*{-2pt}un(\mathcal C_3,\mathcal A)$ and assume
that as such it possesses a left exact coend $H \,{\in}\,
\mathcal Lex(\mathcal C_3,\mathcal A)$, with dinatural transformation $$ 
  h_u(w):\quad G(u,u,w) \to H(w) \equiv \oint^{u\in\mathcal C_2}\! G(u,u,w) \,.
  $$ 
Then the composition
  $$ 
  i_{x,u}(w):\quad F(u,x,x,v,w) \xrightarrow{~j_x(u,u,w)~} G(u,u,w)
  \xrightarrow{~h_u(w)~} H(w)
  $$ 
constitutes the dinatural transformation of a left exact coend $\oint^{x,u}\! F(u,x,x,v,w)$ 
of $F$, with $F$ regarded as a functor $(\mathcal C_1{\btimes}\,\mathcal C_2)^{\mathrm{op}}_{}
\Times (\mathcal C_1{\btimes}\,\mathcal C_2) \,{\xrightarrow{~~\,}}\,
\mathcal Lex(\mathcal C_3,\mathcal A)$.
In short, there is a {\rm(}necessarily unique{\rm)} iso\-mor\-phism
  $$  
  \oint^{x\btimes u\,\in\,\mathcal C_1\btimes\mathcal C_2}\!\!\! F(u,x,x,u,w)
  \,\xrightarrow{~\,\cong\,~}\,
  \oint^{u\in\mathcal C_2}\!\!\! \int^{x\in\mathcal C_1}\!\! F(u,x,x,u,w)
  $$  
of multiple coends. 
\end{prop}

\begin{proof}
For any $f \,{\in}\, \Hom_{\mathcal C_1}(y,x)$ and $g \,{\in}\, \Hom_{\mathcal C_2}(v,u)$ and 
any $w \,{\in}\, \mathcal C_3$ consider the diagram
  $$ 
  \begin{tikzcd}
      F(u,x,y,v,w) \ar{rr}{F(u,f,y,v,w)} \ar{d}{F(u,x,f,v,w)}
  &~& F(u,y,y,v,w) \ar{rr}{F(g,y,y,v,w)} \ar{d}{j_y(u,v,w)}
  &~& F(v,y,y,v,w) \ar{d}{j_y(v,v,w)}
  \\
      F(u,x,x,v,w) \ar{rr}{~~j_x(u,v,w)}     \ar{d}{F(u,y,x,g,w)}
  &~& G(u,v,w)     \ar{rr}{G(g,v,w)}     \ar{d}{G(u,g,w)}
  &~& G(v,v,w)     \ar{d}{h_v(w)}
  \\
      F(u,x,x,u,w) \ar{rr}{~~j_x(u,u,w)}   
  &~& G(u,u,w)     \ar{rr}{~~h_u(w)}
  &~& H(w)
  \end{tikzcd}
  $$ 
in which we use the shorthand $u$ for $\id_u$ etc. The top left and bottom right squares
of this diagram commute by the dinaturality of the families $j$ and $h$, respectively, while
the off-diagonal squares commute because of naturality of $j(u,v,w)$ in the parameters $u$
and $v$. Hence
the outer square commutes, thus establishing the dinaturalness of the family $i_{x,u}$.
 \\
Assume now that $k$ is any dinatural transformation from the functor $F$ to a left exact functor
$K \,{\in}\, \mathcal Lex(\mathcal C_3,\mathcal A)$, meaning that also the outer square in the diagram
  $$ 
  \hspace*{-1.4em} \begin{array}{l}
  \begin{tikzcd}[column sep=5.6ex]
  ~~~~~~F(u,x,y,v,w) \ar{rr}{~~F(u,f,y,v,w)} \ar{d}{F(u,x,f,v,w)}
  &~& F(u,y,y,v,w)   \ar{rr}{F(g,y,y,v,w)}   \ar{d}{j_y(u,v,w)}
 %%% &~& F(v,y,y,v,w)~~~~~~~\ar{d}[left]{j_y(v,v,w)}
 %%% \ar[controls={+(0.1,-0.2) and +(3.7,2.9)},xshift=27pt,yshift=9pt]{ddd} [xshift=-2pt]{k_{v,y}(w)} 
     % for arxiv:
     &~& F(v,y,y,v,w)~~\ar{d}[left]{j_y(v,v,w)}
     \ar[in=30, out=353]{ddd}{k_{v,y}(w)} 
  \\
  ~~~~~~F(u,x,x,v,w) \ar{rr}{~~~j_x(u,v,w)} \ar{d}{F(u,x,x,g,w)}
  &~& G(u,v,w)       \ar{rr}{G(g,v,w)}      \ar{d}{G(u,g,w)}
  &~& G(v,v,w) ~~~~  \ar{d}[left]{h_v(w)}
  \\
  ~~~~~~F(u,x,x,u,w) \ar{rr}{~~~j_x(u,u,w)}
 %%% \ar[controls={+(0.0,-1.6) and +(-1.1,-0.1)},xshift=-5pt,yshift=2pt]{rrrrd}[yshift=1pt]{k_{u,x}(w)} 
     % for arxiv:
     \ar[in=188, out=320]{rrrrd}{k_{u,x}(w)} 
  &~& G(u,u,w)       \ar{rr}{~~h_u(w)}
  &~& H(w) ~~
  \\
 %%% &~&~&~& \begin{picture}(0,0) \put(-3,3){$K(w)$} \end{picture}
     % for arxiv:
     &~&~&~& K(w)
  \end{tikzcd}
  \\[-27pt]~ \end{array}
  $$ 
commutes. Temporarily restricting attention to the case $v \,{=}\, u$ and 
$g \,{=}\, \id_u$ and suppressing the part involving $H(w)$, this diagram collapses to 
  $$ 
  \begin{tikzcd}[row sep=4.4ex]
  ~~~~~~F(u,x,y,u,w) \ar{rr}{~~F(u,f,y,u,w)} \ar{d}[left]{F(u,x,f,u,w)}
  && F(u,y,y,u,w) \ar{d}[left]{j_y(u,u,w)} \ar[bend left]{rdd}[xshift=-3pt]{k_{u,y}(w)} & ~
  \\
  ~~~~~~F(u,x,x,u,w) \ar{rr}{j_x(u,u,w)}
  \ar[bend right=14]{rrrd}[xshift=-26pt,yshift=5pt]{k_{u,x}(w)}
  && G(u,u,w)  ~ & ~
  \\
  &&& K(w)
  \end{tikzcd}
  $$ 
By the universal property of the coend $G$ there then exists a unique morphism
$\varphi_u(w)$ such that also the triangles in the diagram 
  $$ 
  \begin{tikzcd}[row sep=4.4ex]
  ~~~~~~F(u,x,y,u,w) \ar{rr}{~~F(u,f,y,u,w)} \ar{d}[left]{F(u,x,f,u,w)}
  && F(u,y,y,u,w) \ar{d}[left]{j_y(u,u,w)} \ar[bend left]{rdd}[xshift=-3pt]{k_{u,y}(w)} & ~
  \\
  ~~~~~~F(u,x,x,u,w) \ar{rr}{j_x(u,u,w)}
  \ar[bend right=14]{rrrd}[xshift=-26pt,yshift=5pt]{k_{u,x}(w)}
  && G(u,u,w) \ar[dashed]{rd}[xshift=-3pt]{\varphi_u(w)} ~ & ~
  \\
  &&& K(w)
  \end{tikzcd}
  $$ 
commute. Returning to the general case we thus obtain a commuting diagram
  $$  
  \hspace*{-1.4em} \begin{array}{l}
  \begin{tikzcd}
  ~~~~~~F(u,x,y,v,w) \ar{rr}{~~F(u,f,y,v,w)} \ar{d}{F(u,x,f,v,w)}
  &~& F(u,y,y,v,w)   \ar{rr}{F(g,y,y,v,w)}    \ar{d}{j_y(u;v,w)}
 %%% &~& F(v,y,y,v,w)~~~~~~~\ar{d}[left]{j_y(v,v,w)}
 %%% \ar[controls={+(0.1,-0.2) and +(3.7,2.9)},xshift=27pt,yshift=9pt]{ddd} [xshift=-2pt]{k_{v,y}(w)} 
     % for arxiv:
     &~& F(v,y,y,v,w)~ \ar{d}[left]{j_y(v,v,w)}
     \ar[in=5, out=343]{ddd}{k_{v,y}(w)} 
  \\
  ~~~~~~F(u,x,x,v,w) \ar{rr}{~~~j_x(u;v,w)} \ar{d}{F(u,x,x,g,w)}
  &~& G(u;v,w) ~~    \ar{rr}{G(g;v,w)}      \ar{d}{G(u;g,w)}
 %%% &~& G(v;v,w) ~~~~  \ar{d}[left]{h_v(w)}
 %%% \ar[controls={+(0.1,-0.1) and +(1.9,2.6)},xshift=14pt,yshift=10pt]{dd} [xshift=-2pt]{\varphi_v(w)}
     % for arxiv:
     &~& G(v;v,w) ~~ \ar{d}[left]{h_v(w)}
     \ar[in=70, out=333]{dd}{\varphi_v(w)}
  \\
  ~~~~~~F(u,x,x,u,w) \ar{rr}{~~~j_x(u;u,w)}
 %%% \ar[controls={+(0.0,-1.6) and +(-1.1,-0.8)},xshift=-4pt,yshift=1pt]{rrrrd} [yshift=1pt]{k_{u,x}(w)} 
     % for arxiv:
     \ar[in=188, out=320]{rrrrd}{k_{u,x}(w)} 
 %%% &~& G(u;u,w) ~~    \ar{rr}{~~h_u(w)}
 %%% \ar[controls={+(0.0,-0.5) and +(-0.6,0.8)},xshift=-4pt,yshift=3pt]{rrd} [swap]{\varphi_u(w)}
     % for arxiv:
     &~& G(u;u,w) ~ \ar{rr}{~~h_u(w)}
     \ar[in=155, out=320]{rrd} [swap]{\varphi_u(w)}
  &~& H(w) ~~
  \\
 %%% &~&~&~& \begin{picture}(0,0) \put(-2,3){$K(w)$} \end{picture}
     % for arxiv:
     &~&~&~& K(w)
  \end{tikzcd}
  \\[-19pt]~ \end{array}
  $$ 
for any $f \,{\in}\, \Hom_{\mathcal C_1}(y,x)$, $g \,{\in}\, \Hom_{\mathcal C_2}(v,u)$ and $w \,{\in}\, \mathcal C_3$.
Next we invoke the universal property of $H$ as a left exact coend of $G$ (i.e., that the 
family $h$ is universal among all dinatural transformations from $G$ to left exact functors)
to conclude that there is a unique morphism $\psi(w)$ such that also the triangles in
  $$ 
  \begin{tikzcd}[column sep=5.6ex]
  G(u;v,w)   \ar{rr}{G(g;v,w)}  \ar{d}[xshift=-40pt]{G(u;g,w)}
  && G(v;v,w) \ar[bend left]{rdd}[xshift=-3pt]{\varphi_v(w)}  \ar{d}[left]{h_v(w)}
  \\
  G(u;u,w) \ar{rr}{h_u(w)} \ar[bend right=14]{rrrd}[xshift=-26pt,yshift=5pt]{\varphi_u(w)}
  && H(w)  \ar[dashed]{rd}[xshift=-3pt]{\psi(w)}
  \\
  &&& K(w)
  \end{tikzcd}
  $$ 
commute. Taken together, it follows that, given commutativity of the outer and inner squares of
  $$ 
  \begin{tikzcd}
  F(u,x,y,v,w)  \ar{rrrr}{F(g,f,y,v,w)~} \ar{d}[left]{F(u,x,f,g,w)}
  &&&& F(v,y,y,v,w) \ar[bend left]{rdd}[xshift=-2pt]{k_{v,y}(w)}
  \ar{d}[left, xshift=1pt]{h_v(w)\,\circ j_y(v;v,w)}
  \\
  F(u,x,x,u,w) \ar{rrrr}[yshift=-14pt]{~~~h_u(w)\,\circ\,j_x(u;u,w)}
  \ar[bend right=10]{rrrrrd}[xshift=-22pt,yshift=2pt]{k_{u,x}(w)}
  &&&& H(w)  \ar[dashed]{rd}{}
  \\
  &&&&& K(w)
  \end{tikzcd}
  $$ 
for all $f \,{\in}\, \Hom_{\mathcal C_1}(y,x)$, $g \,{\in}\, \Hom_{\mathcal C_2}(v,u)$ and 
$w \,{\in}\, \mathcal C_3$, there is a unique morphism from $H(w)$ to $K(w)$, namely
$\psi(w)$, that also makes the triangles in the diagram commute. This shows the required
universal property of $H(w)$ as a left exact coend of $F$ -- i.e., that the family 
$i \,{=}\, j \,{\circ}\, h$ is universal among all dinatural transformations from $F$ 
to left exact functors -- and thus completes the proof.
\end{proof}

Interchanging the roles of $\mathcal C_1$ and $\mathcal C_2$ in Proposition \ref{thm:lyub11:B.2.}
one arrives at

\begin{coro}\label{coro:gzFubini}
Under the assumptions of Proposition {\rm \ref{thm:lyub11:B.2.}} there is an isomorphism
  \begin{equation}
  \label{fubiniX2}
  \oint^{u\in\mathcal C_2}\!\!\! \int^{x\in\mathcal C_1}\!\! F(u,x,x,u,w)
  \,\xrightarrow{~\,\cong\,~}\, \oint^{x\in\mathcal C_1}\!\!\! \int^{u\in\mathcal C_2}\!\! F(u,x,x,u,w)
  \end{equation}
of iterated coends.
\end{coro}

Or, expressed at greater length: The two objects on the left and right hand sides of 
\eqref{fubiniX2} are isomorphic, and among all isomorphisms between them there is a 
unique one that is compatible with their respective coend structures.

\vskip 3,6em

\noindent
{\sc Acknowledgements:}\\[.3em]
We are grateful to Tobias Ohrmann for helpful comments on the manuscript.
 \\
JF is supported by VR under project no.\ 621-2013-4207.
CS is partially supported by the Collaborative Research Centre 676 ``Particles,
Strings and the Early Universe - the Structure of Matter and Space-Time'', by the RTG 1670
``Mathematics inspired by String theory and Quantum Field Theory'' and by the DFG Priority
Programme 1388 ``Representation Theory''.

%%%%%%%%%%%%%%%%%%%%%%%%%%%%%%%%%%%%%%%%%%%%%%%%%%%%%%%%%%%%%%%%%%%%%%%%

\newpage

 \newcommand\wb{\,\linebreak[0]} \def\wB {$\,$\wb}
 \newcommand\Bi[2]    {\bibitem[#2]{#1}}
 \newcommand\Epub[2]  {{\em #2}, {\tt #1}}
 \newcommand\inBo[8]  {{\em #8}, in:\ {\em #1}, {#2}\ ({#3}, {#4} {#5}), p.\ {#6--#7} }
 \newcommand\inBO[9]  {{\em #9}, in:\ {\em #1}, {#2}\ ({#3}, {#4} {#5}), p.\ {#6--#7} {\tt [#8]}}
 \newcommand\J[7]     {{\em #7}, {#1} {#2} ({#3}) {#4--#5} {{\tt [#6]}}}
 \newcommand\JO[6]    {{\em #6}, {#1} {#2} ({#3}) {#4--#5} }
 \newcommand\JP[7]    {{\em #7}, {#1} ({#3}) {{\tt [#6]}}}
 \newcommand\BOOK[4]  {{\em #1\/} ({#2}, {#3} {#4})}
 \newcommand\PhD[2]   {{\em #2}, Ph.D.\ thesis #1}
 \newcommand\Prep[2]  {{\em #2}, preprint {\tt #1}}

\def\adma  {Adv.\wb Math.}
\def\coma  {Con\-temp.\wb Math.}
\def\comp  {Com\-mun.\wb Math.\wb Phys.}
\def\jktr  {J.\wB Knot\wB Theory\wB and\wB its\wB Ramif.}
\def\joal  {J.\wB Al\-ge\-bra}
\def\jopa  {J.\wb Phys.\ A}
\def\jpaa  {J.\wB Pure\wB Appl.\wb Alg.}
\def\nupb  {Nucl.\wb Phys.\ B}
\def\tams  {Trans.\wb Amer.\wb Math.\wb Soc.}

%%%%%%%%%%%%%%%%%%%%%%%%%%%%%%%%%%%%%%%%%%%%%%%%%%%%%%%%%%%%%%%%%%%%%%%%

\end{document}